\documentclass[12pt, twosided]{article}

\usepackage{enumerate}
\usepackage{amsmath}
\usepackage{amssymb,latexsym}
\usepackage{amsthm}
\usepackage{color}
\usepackage{graphicx}
\usepackage{hyperref}
\usepackage{amscd}
\usepackage[normalem]{ulem}

\DeclareMathOperator{\Tr}{Tr}
\DeclareMathOperator{\N}{N}

\title{Intersecting families of graphs of functions over a finite field}
\author{Angela Aguglia \footnote{Dipartimento di Meccanica, Matematica e Management, Politecnico di Bari, Via Orabona 4, I-70125 Bari, Italy; \tt{angela.aguglia@poliba.it}}
	\and Bence Csajb\'ok \footnote{Dipartimento di Meccanica, Matematica e Management, Politecnico di Bari, Via Orabona 4, I-70125 Bari, Italy; \tt{bence.csajbok@poliba.it}}
	\and Zsuzsa Weiner \footnote{ELKH--ELTE Geometric and Algebraic Combinatorics Research Group, 1117 Budapest, P\'azm\'any P.\ stny.\ 1/C, Hungary and 	Prezi.com,		H-1065 Budapest, Nagymez\H o utca 54-56, Hungary; \tt{zsuzsa.weiner@gmail.com}}
	\thanks{The second and the third author acknowledge the support of the National Research, Development and Innovation Office – NKFIH, grant no. K 124950. This work was supported by the Italian National Group for Algebraic and Geometric Structures and their Applications (GNSAGA– INdAM).}}
\date{}

\newcommand{\cA}{{\mathcal A}}
\newcommand{\cB}{{\mathcal B}}

\newcommand{\cC}{{\mathcal C}}

\newcommand{\cM}{{\mathcal M}}

\newcommand{\cF}{{\mathcal F}}
\newcommand{\cP}{{\mathcal P}}

\newcommand{\cQ}{{\mathcal Q}}
\newcommand{\cS}{{\mathcal S}}

\newcommand{\F}{{\mathbb F}}

\newcommand{\GF}{\hbox{{\rm GF}}}

\newtheorem{theorem}{Theorem}[section]

\newtheorem{lemma}[theorem]{Lemma}
\newtheorem{corollary}[theorem]{Corollary}
\newtheorem{definition}[theorem]{Definition}

\newtheorem{result}[theorem]{Result}
\newtheorem{example}[theorem]{Example}

\newtheorem{remark}[theorem]{Remark}

\DeclareMathOperator{\PG}{{PG}}
\DeclareMathOperator{\AG}{{AG}}

\begin{document}
	\maketitle
	
	\begin{abstract}
		Let $U$ be a set of polynomials of degree at most $k$ over $\F_q$, the finite field of $q$ elements. Assume that $U$ is an intersecting family, that is, the graphs of any two of the polynomials in $U$ share a common point.
		Adriaensen proved that the size of $U$ is at most $q^k$ with equality if and only if $U$ is the set of all polynomials of degree at most $k$ passing through a common point. In this manuscript, using a different, polynomial approach, we prove a stability version of this result, that is, the same conclusion holds if $|U|>q^k-q^{k-1}$. We prove a stronger result when $k=2$.
		
		For our purposes, we also prove the following results. If the set of directions determined by the graph of $f$ is contained in an additive subgroup of $\mathbb{F}_q$, then the graph of $f$ is a line. If the set of directions determined by at least $q-\sqrt{q}/2$ affine points is contained in the set of squares/non-squares plus the common point of either the vertical or the horizontal lines, then up to an affinity the point set is contained in the graph of some polynomial of the form $\alpha x^{p^k}$.
		
		
	\end{abstract}

	\section{Introduction}
	In 1961, Erd\H os et al.\ \cite{Erdos} proved that if $F$ is a $k$-uniform intersecting family of
	subsets of an $n$-element set $X$, then $|F| \leq  \binom{n-1}{k-1}$ when $2k \leq  n$. Furthermore, they proved that if $2k+1 \leq n$, then equality holds if and only if $F$ is the family of all subsets containing a fixed element $x \in X$.
	There are several versions of the Erd\H os-Ko-Rado theorem.  For a survey of this type of results, see
	\cite{Frankl}, \cite{Tokushige} or \cite{BlokhuisAtAl}.
	
	In this manuscript, we investigate an Erd\H os-Ko-Rado type problem for graphs of functions over a finite field.
	The idea of this work comes from the recent manuscript \cite{Sam} by Adriaensen, where the author studies intersecting families of ovoidal circle geometries and, as a consequence, of graphs of functions over a finite field.

	\begin{definition}
		If $f$ is an $\F_q \rightarrow \F_q$ function, then the graph of $f$ is the affine $q$-set:
		\[U_f=\{(x,f(x)) : x \in \F_q\}.\]
		The set of directions determined by (the graph of) $f$ is
		\[D_f=\left\{ \frac{f(x)-f(y)}{x-y} : x,y \in \F_q,\, x\neq y \right\}.\]
	\end{definition}
	
	\begin{definition}
		For a family of polynomials, $U$, we say that $U$ is $t$-intersecting if for any two polynomials $f_1, f_2 \in U$, the graphs of $f_1$ and $f_2$ share at least $t$ points, that is,
		\[|\{ (x,f_1(x)) : x \in \F_q\} \cap \{ (x,f_2(x)) : x \in \F_q\}|\geq t.\]
		Instead of $1$-intersecting, we will also use the term ``intersecting''.
		
	\end{definition}
	Note that if $U$  is a $t$-intersecting family of polynomials of degree at most $k$, then also
	\[|\{ (x,f_1(x)) : x \in \F_q\} \cap \{ (x,f_2(x)) : x \in \F_q\}|\leq k\]
	holds for any pairs $f_1,f_2\in U$, since $(x,f_1(x))=(x,f_2(x))$ implies that $x$ is a root of $f_1-f_2$ which has degree at most $k$.

	In this note, we improve a result  due to  Adriaensen {\cite[Theorem 6.2]{Sam}} by using different techniques. Adriaensen's proof goes through association schemes and circle geometries, our proof does not use these, as we rely on two classical results (Results \ref{res1} and \ref{res2}) about polynomials over finite fields.
	
	To be more precise, our main results are the following theorems.
	
	\begin{theorem}
		\label{degtwo}
		Let $U$ be a set of intersecting polynomials of degree  $k\leq 2$ over $\F_q$. Assume that $q \geq 53$, when $q$ is odd and $q \geq 8$ when $q$ is even. If
		\[|U|>q^2-\frac{q\sqrt{q}}{4} + \frac{cq}{8} + \frac{\sqrt q}{8},\]
		where $c=1$ for $q$ even and $c=3$ for $q$ odd,  then the graphs of the functions in $U$ share a common point.
	\end{theorem}
	
	We will prove Theorem \ref{degtwo}	separately for $q$ odd (Theorem \ref{deg2dispari}) and for $q$ even (Theorem \ref{deg2p}).
	For $k>2$, the proof can be finished by induction as in \cite[pg. 33]{Sam}. We obtain the following result.
	
	\begin{theorem}
		\label{general}
		If $U$ is a set of more than $q^{k}-q^{k-1}$ intersecting polynomials over $\F_q$, $q \geq 53$ when $q$ is odd and $q \geq 8$ when $q$ is even, and of degree at most $k$, $k > 1$, then there exist $\alpha,\beta \in \F_q$ such that $g(\alpha)=\beta$ for all $g\in U$. Furthermore, $U$  can be uniquely extended to a family of  $q^k$ intersecting polynomials over $\F_q$ and of degree at most $k$.
	\end{theorem}
	
	
	While finalizing our manuscript, a stronger stability version of the above mentioned result for $k=2$ was published by Adriaensen; see \cite{Sam2}.
	Our proof is different, based on polynomials and hence might be of independent interest.

	\section{Preliminaries}
	
	Throughout this  paper, $q=p^n$ for some prime $p$ and a positive integer $n$.
	The algebraic closure of the finite field $\F_q$ will be denoted by $\overline{\F}_q$.

	The absolute trace function is defined as $\Tr_{q/p} \colon \F_q \rightarrow \F_p$, $\Tr_{q/p}(x)=x+x^p+\ldots+x^{p^{n-1}}$.
	Recall that for $q$ even and $c\neq 0$, $a+bx+cx^2\in \F_q[x]$ has a root in $\F_q$ if and only if $b=0$, or $b\neq 0$ and
	\begin{equation}
		\label{trace}
		\Tr_{q/2}\left( \frac{ac}{b^2}\right)=0.
	\end{equation}
	When $q$ is a square, we will also use the notation $\N \colon \F_{q} \rightarrow \F_{\sqrt{q}}$, $x \mapsto x^{\sqrt{q}+1}$, which is the norm of $x$ over $\F_{\sqrt{q}}$.
	
We will frequently need the following result of Ball, Blokhuis, Brouwer, Storme, Sz\H onyi and Ball.

\begin{result}[Part of \cite{SB, SB2}]
	\label{BBBSSz}
	Let $f$ be an $\F_q \rightarrow \F_q$ function such that $|D_f|\leq(q+1)/2$. Then $U_f$ is affinely equivalent to the graph of a linearised polynomial, that is, a polynomial of the form $\sum_{i=0}^{n-1}a_i x^{p^i} \in \F_q[x]$. 
\end{result}

	\begin{theorem}
		\label{additiveCarlitz}
		Let $U$ denote a proper $\F_p$-subspace of $\F_q$, $q=p^n>2$, $p$ prime and consider a function $\sigma \colon \F_q \rightarrow \F_q$. If the set of directions
		\[D_\sigma=\left\{  \frac{\sigma(x)-\sigma(y)}{x-y} : x,y\in \F_q,\, x\neq y \right\}\]
		is contained in $U$, then $\sigma(x)=ax+b$ for some $a,b\in \F_q$.
	\end{theorem}
	\begin{proof}
		First note that $U$ is contained in an $(n-1)$-dimensional $\F_p$-subspace $V$ and hence $|D_\sigma|\leq p^{n-1}$.
		Then by Result \ref{BBBSSz}, $\sigma(x)=\alpha+g(x)$, where $\alpha\in \F_q$ and $g(x)=\sum_{i=0}^{n-1}b_ix^{p^i}\in \F_q[x]$, thus
		\[D_{\sigma}=\left\{ \frac{g(x)}{x} : x\in \F_q\setminus \{0\}\right\}.\]
		It is well-known that $\beta^{q/p}\prod_{\gamma\in V}(x-\gamma)=\Tr_{q/p}(\beta x)$ for some $\beta\in \F_q\setminus\{0\}$. Next define $f(x):=\beta g(x)$. Then $D_{\sigma} \subseteq V$ implies
		\[\Tr_{q/p}\left(\frac{f(x)}{x}\right)=0\]
		for each $x\in \F_q \setminus \{0\}$.
		To prove the assertion, it is enough to prove that $f(x)$ is linear.
		With $f(x)=\sum_{i=0}^{n-1}a_i x^{p^i}\in \F_q[x]$,
		\[\Tr_{q/p}\left(\frac{f(x)}{x}\right)=\Tr_{q/p}\left(\sum_{i=0}^{n-1}a_i x^{p^i-1}\right)=
		\sum_{j=0}^{n-1} \sum_{i=0}^{n-1}a_i^{p^j}x^{p^{i+j}-p^j},\]
		and because of our assumption, this polynomial vanishes at every element of $\F_q \setminus \{0\}$.
		
		\medskip
		
		\noindent
		{\bf The $p>2$ case.}
		
		If we multiply this polynomial by $x^{1+p+p^2+\ldots+p^{n-1}}$, then
		we obtain
		\[\sum_{j=0}^{n-1} \sum_{i=0}^{n-1}a_i^{p^j}x^{1+p+\ldots+p^{n-1}+p^{i+j}-p^j}\]
		and this polynomial vanishes at every element of $\F_q$.
		As a function, this polynomial remains the same if we consider it modulo $x^{p^n}-x$, so it is the same function as the polynomial  we obtain when we replace
		the exponents $p^{i+j}$ with  $p^{i+j-n}$ for each $i+j\geq n$. Denote this new polynomial with $\tilde{f}$.
		The fact that we multiplied $f$ by $x^{1+p+p^2+\ldots+p^{n-1}}$ ensures that the exponents of $\tilde{f}$ are larger than
		$0$ and smaller than $q$.
		We claim that each monomial has  different degree in $\tilde{f}$.
		It is clear that $\tilde{f}$ is the sum of at most $n^2$ monomials and the set of degrees of these monomials is contained in the set
		\[A:=\{1+p+p^2+\ldots+p^{n-1}+p^{c}-p^d : c,d \in \{0,1,\ldots,n-1\} \}.\]

		Assume that for some
		$c_1,c_2\in \{0,1,\ldots,n-1\}$,
		$d_1,d_2 \in \{0,1,\ldots,n-1\}$ and $(c_1,d_1)\neq (c_2,d_2)$
		
		\[1+p+p^2+\ldots+p^{n-1}+p^{c_1}-p^{d_1}=1+p+p^2+\ldots+p^{n-1}+p^{c_2}-p^{d_2},\]
		or equivalently $p^{c_1}+p^{d_2}=p^{c_2}+p^{d_1}$. Since the base $p$-digits of an integer are uniquely determined,
		this implies $\{c_1,d_2\}=\{c_2,d_1\}$, so either
		$c_1=c_2$ and $d_1=d_2$, or $c_1=d_1$ and $c_2=d_2$. We conclude that two distinct monomials of $\tilde{f}$ have the same degree $d$ if any only if
		$d=1+p+\ldots+p^{n-1}$, that is, when in
		\[1+p+\ldots+p^{n-1}+p^{i+j}-p^j\]
		we have $i=0$.
		
		Note that the degree of $\tilde{f}$ is at most
		\[m:=1+p+\ldots+p^{n-1}+p^{n-1}-1.\]
		Since $p>2$, $m$ is clearly smaller than $q$, but $\tilde{f}$ has $q$ roots (the elements of $\F_q$). Thus it is the zero polynomial, all of its coefficients are zero. The coefficients of $\tilde{f}$ are
		the $p^j$-powers of $a_1,\ldots,a_{n-1}$ and $\Tr_{q/p}(a_0)$. So $f(x)=a_0x$.
		
		\medskip
		
		\noindent
		{\bf The $p=2$ case.}
		
		Note that when $i\neq 0$, then
		\[2^{i+j}-2^j=2^j+2^{j+1}+\ldots+2^{i+j-1}.\]
		If $2^{i+j}-2^j \geq 2^n=q$ then write this number as
		\[2^{i+j}-2^j=2^j+2^{j+1}+\ldots+2^{i+j-1}=
		2^j+2^{j+1}+\ldots+2^{n-1}+2^n(1+\ldots+2^{i+j-1-n}).\]
		Clearly, as $\F_q \rightarrow \F_q$ functions,
		\[x \mapsto x^{2^{i+j}-2^j}\]
		are the same as
		\[x \mapsto x^{2^j+2^{j+1}+\ldots+2^{n-1}}x^{1+\ldots+2^{i+j-1-n}}=\]
		\[x^{1+\ldots+2^{i+j-1-n}+2^j+2^{j+1}+\ldots+2^{n-1}}.\]
		When $2^{i+j}-2^j\geq 2^n=q$, then substitute these exponents in
		\[\Tr_{q/2}\left(\frac{f(x)}{x}\right)=
		\sum_{j=0}^{n-1} \sum_{i=0}^{n-1}a_i^{2^j}x^{2^{i+j}-2^j}\]
		with the exponents
		\[B_{ij}:=1+\ldots+2^{i+j-1-n}+2^j+2^{j+1}+\ldots+2^{n-1}<q,\]
		and denote this new polynomial with $\tilde{f}$. Note that in this case $i+j-1-n<j-1$.
		When $2^{i+j}-2^j<q$, $i\neq 0$, then define
		\[A_{ij}:=2^{i+j}-2^j=2^j+2^{j+1}+\ldots+2^{i+j-1}<q.\]
		If $i=0$ then put $A_{0j}=0$.
		Since base $2$-digits of an integer are uniquely determined, we have the following:
		$(1)$ If $i_1 \neq 0$, then
		$A_{i_1j_1}=A_{i_2j_2}$ iff $(i_1,j_1)=(i_2,j_2)$,
		$(2)$ $A_{0j_1}=A_{0j_2}$ for any pair $(j_1,j_2)$, $(3)$
		$B_{i_1j_1}=B_{i_2j_2}$ iff $(i_1,j_1)=(i_2,j_2)$,
		finally $(4)$ $B_{i_1j_1}=A_{i_2j_2}$ iff $i_1+j_1-n=j_1$ and $j_2=0$ and $i_2=n$ (otherwise $B_{i_1j_1}$ in base $2$ has the form $11\ldots110\ldots011\ldots11$, while $A_{i_2j_2}$ has the form
		$11\ldots1100\ldots00$, a contradiction); but $i_1, i_2<n$.
		It follows that the only exponent which appears in more than one monomial is the $0$.
		The degree of $\tilde{f}$ is at most $q-2$ (obtained in $x^{A_{(n-1)\,1}}$) and it has
		$q-1$ roots (the elements of $\F_q \setminus \{0\}$), so it is the zero polynomial.
		Hence all of its coefficients are zero. These coefficients are the $2^j$-powers of $a_1,\ldots, a_{n-1}$ and $\Tr_{q/2}(a_0)$. It follows that $f(x)=a_0x$.
	\end{proof}

	We will need the following two results regarding functions over finite fields.
	
	\begin{result}[{\cite[Theorem 5.41]{FFBook}}, Weil's bound]
		\label{res1}
		Let $\psi$ be a multiplicative character of $\F_q$ of order $m>1$ and let $f\in \F_{q}[x]$ be a monic polynomial of positive degree that is not an $m$-th power of a polynomial. Let $d$ be the number of distinct roots of $f$ in $\overline{\F}_q$. Then for every $a\in \F_q$ we have
		\[\left| \sum_{c\in \F_q} \psi(a f(c)) \right| \leq (d-1)\sqrt{q}.\]	
	\end{result}
	
	We will also consider polynomials of degree $\sqrt{q}+1$ admitting square values for almost every element of $\F_q$. In this case, the inequality above seems to be useless. In Lemma \ref{shortcut}, we show a way how to derive information from Weil's bound also in this case. When $m=d=2$ then the following, stronger result holds which can be easily proved by counting $\F_q$-rational points of a conic of $\PG(2,q)$:
	
	\begin{result}[{\cite[Exercise 5.32]{Sziklaipoly}}]
		\label{res1b}
		Let $q$ be an odd prime power, $f(x)=ax^2+bx+c \in \F_q[x]$ with $a\neq 0$, and let $\psi$ denote the quadratic character $\F_q \rightarrow \{-1,1,0\}$. Then
		\[\sum_{x\in \F_q} \psi(ax^2+bx+c)\]
		equals $-\psi(a)$ if $b^2-4ac\neq 0$ and $(q-1)\psi(a)$ if $b^2-4ac=0$.
	\end{result}

	
	\medskip
	
	To use Result \ref{res1}, we will need the following.
	
	
	\begin{lemma}
		\label{smalldegree}
		Put $f(x)=a x^{p^k+1}+dx^{p^{k}}+bx+c \in \F_q[x]$, $k\neq 0$.
		If $q$ is odd and $f(x)=g(x)^2$, then $d^{p^k}a=ba^{p^k}$ and $d^{p^k+1}a=ca^{p^k+1}$, or $a=b=d=0$.
	\end{lemma}
	\begin{proof}
		If $a=0$, then $b=d=0$  otherwise the degree of $f$ was odd. Assume $a\neq 0$ and suppose $f(x)=g(x)^2$. Then the roots of $f$ have multiplicities at least $2$ and hence they are also roots of $f'(x)=ax^{p^k}+b=(a^{p^{-k}}x+b^{p^{-k}})^{p^k}$. It follows that $f(x)$ has a unique root, $-(b/a)^{p^{-k}}$, so
		\[f(x)=a(x+\gamma)^{p^k+1}=a(x^{p^k}+\gamma^{p^k})(x+\gamma)=ax^{p^k+1}+\gamma ax^{p^k}+a\gamma^{p^k}x+\gamma^{p^k+1}a,\]
		with $\gamma=(b/a)^{p^{-k}}$.
		It follows that $f$ has the listed properties.
	\end{proof}

	\begin{lemma}
		\label{shortcut}
		If for some odd, square $q>9$ there is a subset $D$ of $\F_q$ of size larger than $q-\sqrt{q}/2+1/2$ such that the $\F_{q}\rightarrow \F_{q}$ function $x \mapsto \ell(x):=ax^{\sqrt{q}+1}+dx^{\sqrt{q}}+bx+c$, $a\neq 0$, has the property that $\ell(x)$ is a square of $\F_{q}$ for each $x\in D$, then $a^{\sqrt{q}}b=d^{\sqrt{q}} a$.
	\end{lemma}
	\begin{proof}
		Suppose $a^{\sqrt{q}}b \neq d^{\sqrt{q}}a$.
		Then the value set of $\ell$ clearly does not change if we replace $x$ with $g(y)=((b/a)^{\sqrt{q}}-(d/a))y-d/a$, since $g$ is a permutation polynomial.
		Also, $C:=g^{-1}(D)$, will have the properties that $|C|> q-\sqrt{q}/2+1/2$ and for each $y\in C$,
		\[f(y):=\ell(g(y))\]
		is a square of $\F_q$.
		One can easily verify $f(y)=\ell(g(y))=\beta y^{\sqrt{q}+1}+\beta y+\alpha$, where
		\[\beta=\frac{(a^{\sqrt{q}}b-ad^{\sqrt{q}})^{\sqrt{q}+1}}{a^{2\sqrt{q}+1}}\]
		and $\alpha=c-bd/a$. We will show that this is not possible.
		
		Since the norm $x \mapsto \N(x)$ takes $(\sqrt{q}-1)$ distinct non-zero values in $\F_q$, and $(\sqrt{q}-1) - (\sqrt{q}/2 - 1/2) \geq 2$ (here we use $q>9$ square), we may take $t_1,t_2\in \F_{q}\setminus \{0\}$  such that $\N(t_1)\neq \N(t_2)$ and $\N(t_1), \N(t_2) \notin \{ \N(d) : d \in \F_q \setminus C\}$.
		We show that if $f(x)$ is a square for each $x\in C$ then also the polyomials
		\[f(t_1 y^{\sqrt{q}-1})= \N(t_1) \beta y^{q-1}+\beta t_1 y^{\sqrt{q}-1}+\alpha\in \F_{q}[y],\]
		\[f(t_2 y^{\sqrt{q}-1})= \N(t_2) \beta y^{q-1}+\beta t_2 y^{\sqrt{q}-1}+\alpha\in \F_{q}[y],\]
		have only square values for each $y\in C$.
		Indeed, this follows from the fact that $\N(t_i y^{\sqrt{q}-1}) \notin \{ \N(d) : d \in \F_q \setminus C\}$ and hence $t_i y^{\sqrt{q}-1} \in C$ for $i=1,2$.
		
		Then the polynomials
		\[G_1(y):=\N(t_1) \beta+\beta t_1 y^{\sqrt{q}-1}+\alpha \in \F_{q}[y],\]
		\[G_2(y):=\N(t_2) \beta+\beta t_2 y^{\sqrt{q}-1}+\alpha \in \F_{q}[y],\]
		take only square values on the non-zero elements of $C$.
		Denote by $\psi$  the multiplicative character of $\F_q$ of order two. The polynomial
		$G_i$ has at most $\sqrt{q}-1$ roots, and $\psi(G_i(x))=1$ for every element $x$ of $C\setminus\{0\}$ if $x$ is not a root of $G_i$.
		Define $\varepsilon$ to be $\psi(G_i(0))$ if $0\in C$ and to be $0$ otherwise. Then
		$|C\setminus \{0\}|-(\sqrt{q}-1)+\varepsilon \leq \sum_{x\in C} \psi(G_i(x))$.
		On the other hand, $-(q-|C|) \leq \sum_{x\in \F_q \setminus C} \psi(G_i(x))$, and hence
		\[2|C|-q-\sqrt{q}-1\leq\left| \sum_{y\in \F_q} \psi(G_i(y)) \right|.\]		
		Since
		\[(\sqrt{q}-2)\sqrt{q}<  2|C|-q-\sqrt{q}-1,\]		
		by Result \ref{res1} (with $m=2$) this can only happen if $G_i=g_i^2$ for some polynomials $g_i$, $i=1,2$. Then the roots of $G_i$ (in the algebraic closure of $\F_q$) are multiple roots of $G_i$ and hence also roots of $\gcd(G_i,G_i')$. The only root of $G_i'$ is $0$, thus $G_i(0)=0$ and hence $\N(t_i)\beta+\alpha=0$. Since $a^{\sqrt{q}}b \neq d^{\sqrt{q}}a$, we have $\beta\neq 0$. It follows that $\N(t_i)=-\alpha/\beta$ for $i=1,2$, a contradiction because of the choice of $t_1$ and $t_2$.
	\end{proof}
	
	The next example shows that $\ell(x)=ax^{\sqrt{q}+1}+dx^{\sqrt{q}}+bx+c$ can have only square values if $a^{\sqrt{q}}b=d^{\sqrt{q}} a$ holds.
	
	\begin{example}
		For $t,r\in \F_q$, the polynomial
		\[f(x)=r^{\sqrt{q}+1}x^{\sqrt{q}+1}+r^{\sqrt{q}}tx^{\sqrt{q}}+rt^{\sqrt{q}}x+t^{\sqrt{q}+1}=(t+rx)^{\sqrt{q}+1}\]
		has only square values in $\F_q$.
	\end{example}

We will need a generalisation of the following result by G\"olo\u glu and McGuire. 

\begin{result}[{\cite[Theorem 1.2]{McGuire}}]
		\label{res2}
		Let $q$ be odd and consider a non zero polynomial $L(x)=\sum_{i=0}^{n-1}a_i x^{p^i} \in \F_q[x]$. Denote by $\square_q$ the set of non-zero squares in $\F_q$. Then
\[\mathrm{Im}\,\left( \frac{L(x)}{x}\right) \subseteq \square_q \cup \{0\}\]
if and only if $L(x)=ax^{p^d}$ for some $a \in \square_q$ and $0\leq d \leq n$.
	\end{result}

\begin{definition}
If $U$ is a point set of $\AG(2,q)$, then the set of directions defined by $U$ is
\[D_U=\left\{\left( \frac{a-b}{c-d} \right) : (a,b),(c,d)\in U, (a,b)\neq (c,d)\right\}.\]
(If the denominator is zero then $\left(\frac{a-b}{0}\right)=(\infty)$, the ideal point of vertical lines.)
\end{definition}

In the proof of Theorem \ref{teonew} the following result of Sz\H onyi will be crucial.

\begin{result}[{\cite[Theorem 4 and Proposition 6]{TS}}]
	\label{dirszonyi}
	Let $U$ be a point set of $\AG(2,q)$ of size at least $q-\sqrt{q}/2$ and let $D_U$ be the set of directions determined by $U$.
	\begin{enumerate}
		\item If $U$ determines less than $(q+1)/2$ directions, then $U$ can be extended to a $q$-set determining the same set of directions as $U$.
		\item If $U$ determines exactly $(q+1)/2$ directions, one of them is $(\infty)$ and there is no point $P\in \AG(2,q) \setminus U$ such that $U\cup \{P\}$ determines the same set of directions as $U$, then the $(q+1)/2$-set
		\[\{d\in \F_q,\, (d)\notin D_U\}\]
		is the set of $Y$ coordinates of the points of an irreducible conic $\cC$ of $\AG(2,q)$ and the direction $(0)$ is an internal point of $\cC$.
	\end{enumerate}
\end{result}

\begin{remark}
	\label{unicity}
	By {\cite[Remark 3.3]{TS2}} a blocking set of size at most $2q$ contains a unique minimal blocking set. Let $U$ denote an affine point set of size at least $q-\sqrt{q}/2$ such that $U$ determines less than $(q+1)/2$ directions. 
	Assume that $\cP$ and $\cP'$ are two affine point sets of size $q-|U|$ which extend $U$ to a $q$-set determining the same set of directions as $U$. Then 
	$\cB:=U \cup \cP \cup \cP' \cup D_U$ is a blocking set of size at most 
	$\lfloor q+\sqrt{q}/2+(q+1)/2 \rfloor \leq 2q$ and hence $\cB$ contains a unique minimal blocking set. But both $U \cup \cP \cup D_U$ and $U \cup \cP' \cup D_U$ are minimal blocking sets and this proves $\cP=\cP'$, that is, the unicity of the extension of $U$ in Result \ref{dirszonyi}.
\end{remark}

\begin{lemma}\label{lemnew}
Let $\cS$ denote the set of non-zero squares or non-squares in $\GF(q)$. If the set of $Y$ coordinates of the points of an irreducible conic $\cC$ of $\AG(2,q)$, $q \geq 53$ odd, is contained in $\cS \cup \{0\}$ then $\cC$ is a parabola with equation
\[Y=\alpha (a'X+b'Y+c')^2,\]
where $\alpha \in \cS$.
\end{lemma}
\begin{proof}
Note that horizontal translations of $\cC$ does not affect the properties that we are examining, so after substituting $X$ by $X-\beta$ for a suitable $\beta \in \F_q$ we may assume that $(0,0)$ is not a point of $\cC$ and hence the equation of the conic is
\[aX^2+bXY+cY^2+dX+eY+1=0.\]
The direction $(0)$ cannot be a point of the projective extension of $\cC$ since otherwise we would get at least $q-1>(q+1)/2$  different $Y$ coordinates. It follows that there are at most $2$ horizontal lines meeting $\cC$ in $1$ point and at least $(q-3)/2$ horizontal lines meeting $\cC$ in $2$ points. Fix some $\alpha \in \cS$. At least $(q-5)/2$ horizontal lines meet $\cC$ in $2$ points $(A_i,\alpha B_i^2)$ and $(A'_i,\alpha B_i^2)$ with $B_i\neq 0$; and $\cC$ has at most $2$ points on the $X$ axis.
Next define the quartic $\cQ$ (which might as well be of smaller degree if $c=0$):
\[aX^2+\alpha bXY^2+\alpha^2 cY^4+dX+\alpha eY^2+1=0.\]
Points of $\cC$ on the $X$ axis are points of $\cQ$ as well, and if $(A_i,\alpha B_i^2)$, $(A'_i,\alpha B_i^2)$, $B_i\neq 0$, were two points of $\cC$ then
$(A_i,\pm B_i)$, $(A'_i,\pm B_i)$ are $4$ points of $\cQ$. It follows that 
\[\cA=\{(x,y), (x,-y) : (x,\alpha y^2) \in \cC\}\]
is a subset of the point set of $\cQ$ of size at least $2+2(q-5)=2q-8$.
Note that $2q-8 > q+1+6\sqrt{q}$ (here we use $q\geq 53$). It follows instantly from the Hasse-Weil bound that $\cQ$ cannot be an irreducible cubic or an irreducible quartic.

First we show that $\cQ$ cannot be a quartic curve which is the product of an irreducible cubic and a line. 
Vertical and horizontal lines meet $\cA$ in at most $2$ points
and hence if such a line would be a factor of $\cQ$ then the remaining at least $2q-10$ points of $\cA$ should lie on the cubic, a contradiction again by the Hasse-Weil bound now applied to cubic curves. 
Similarly, if $Y=mX+n$ was a factor of $\cQ$, for some $m\neq 0$, then
\[aX^2+\alpha bX(mX+n)^2+\alpha^2 c(mX+n)^4+dX+\alpha e(mX+n)^2+1\]
was the zero polynomial. The coefficient of $X^4$ is $\alpha^2 c m^4$, so $c$ has to be zero but then $\cQ$ is not a quartic curve, a contradiction. 


From now on we may assume that
\[aX^2+\alpha bXY^2+\alpha^2 cY^4+dX+\alpha eY^2+1=F\cdot G,\]
where $F$ and $G$ are of degree at most $2$. Put
\[F=(a_1Y^2+b_1X+c_1Y+1+d_1X^2+e_1XY),\] \[G=(a_2Y^2+b_2X+c_2Y+1+d_2X^2+e_2XY).\]
In $F\cdot G$ the coefficient of $Y$ is $c_1+c_2$, while it is $0$ in the equation of $\cQ$, so clearly $c_2=-c_1$ and we will use this from now on. The coefficient of $X^4$ is $d_1d_2$ and it has to be zero, so from now on we may assume $d_1=0$. Then the coefficient of $X^3$ is $d_2b_1$ and it has to be zero.

In this paragraph assume \fbox{$d_2\neq 0$}, then $b_1=0$ and the coefficient of $X^3Y$ is $d_2e_1$ so $e_1=0$. The coefficient of $YX^2$ is $c_1d_2$, so $c_1=0$. Then the coefficient of $XY$ is $e_2$, so $e_2=0$. The coefficient of $X^2Y^2$ is $d_2a_1$, so $a_1=0$. We arrived to the conclusion that the equation of $\cQ$ is  $a_2Y^2+b_2X+1+d_2X^2$. It follows that the equation of $\cC$ is $Y=-\alpha (b_2X+1+d_2X^2)/a_2$. Then by Result \ref{res1b}, $-(b_2X+1+d_2X^2)/a_2=(a'X+b')^2$ for some $a',b'\in \F_q$ and this finishes the proof of the $d_2\neq 0$ case.

Now assume \fbox{$d_2=0$} (recall also $d_1=0$). Then the coefficient of $X^2Y^2$ is $e_1e_2$. We may assume $e_1=0$. The coefficient of $X^2Y$ is $e_2b_1$. First assume \underline{$e_2\neq 0$}, so $b_1=0$.
Then the coefficient of $Y^3X$ is $a_1e_2$, so $a_1=0$. Then the coefficient of $Y^3$ is $c_1a_2$. We cannot have $c_1=0$ since then the coefficient of $XY$ would be $e_2\neq 0$, so $a_2=0$. The coefficient of $XY$ is $c_1b_2+e_2=0$ and hence the equation of $\cQ$ is: $(1+c_1Y)(1-c_1Y)(1+b_2X)$, a contradiction since vertical and horizontal lines contain at most $2$ points of $\cA$.
So \underline{$e_2=0$} and from now on we may assume that $\cQ$ has equation
\[(a_1Y^2+b_1X+c_1Y+1)(a_2Y^2+b_2X-c_1Y+1)=0.\]
The fact that $Y^3$ and $XY$ should have zero coefficient yields $a_1=a_2$ and $b_1=b_2$, or $c_1=0$.
In the former case the equation of $\cQ$ is
\[(a_1Y^2+b_1X+c_1Y+1)(a_1Y^2+b_1X-c_1Y+1)=0,\]
so $\cC$ had equation
\[(\alpha^{-1} a_1Y+b_1X+1)^2-\alpha^{-1}c_1^2Y=0,\]
which proves the assertion.

In the latter case the equation of $\cQ$ is
\[(a_1Y^2+b_1X+1)(a_2Y^2+b_2X+1)=0,\]
so $\cC$ had equation
\[(a_1Y/\alpha+b_1 X+1) (a_2 Y/\alpha+b_2 X+1)=0,\]
a contradiction since $\cC$ was irreducible.
\end{proof}

The following can be considered as a generalisation of Result \ref{res2}.

\begin{theorem}\label{teonew}
	Let $U$ denote a point set of $\AG(2,q)$, $q\geq 53$ odd, of size at least $q-\sqrt{q}/2$. Let $\cS$ denote the set of non-zero squares or non-squares in $\F_q$ and let $(d)$ denote one of the directions $(0)$ or $(\infty)$. If $D_U$ is contained in $\{(s) : s \in \cS\} \cup \{(d)\}$, then $U$ is affinely equivalent to a subset of the graph of a function of the form
	\[f(x)=\alpha x^{p^k},\]
	where $\alpha \in \cS$.
\end{theorem}
\begin{proof}
If $|D_U|<(q+1)/2$ then by Result \ref{dirszonyi} $U$ can be extended to a $q$-set $U'$ determining the same set of directions. According to Result \ref{BBBSSz} $U'$ is affinely equivalent to the graph of a linearised polynomial $f$. Then Result  \ref{res2} shows that $f$ has the requested form.

Now assume $|D_U|=(q+1)/2$. If $(d)=(0)$, then apply the affinity $\varphi \colon (x:y:z)\mapsto (y:x:z)$.
Clearly, $D_{U^{\varphi}}=(D_U)^{\varphi}$ and $U$ can be extended if and only if $U^{\varphi}$ can be extended.
We have $(0)^\varphi=(\infty)$ and
if $m\neq 0$ then $(m)^\varphi=(1/m)$, so
\[\{(s)^\varphi : s \in S\}=\{(s) : s \in S\}.\]
By Result \ref{dirszonyi}, if $U^\varphi$ (or $U$, if $(d)=(\infty)$) cannot be extended, then the set of non-zero squares or non-squares together with the zero  equals the set of $Y$ coordinates of an irreducible affine conic $\cC$ and $(0)$ is an internal point of $\cC$. Then the line at infinity is not a tangent to $\cC$, thus $\cC$ is not a parabola (and not a hyperbola because then the size of the set of $Y$ coordinates would be $(q-1)/2$; but we don't need this) but this is not possible because of the Lemma \ref{lemnew}. It follows that $U$ can be extended to a $q$-set determining the same set of directions as $U$ and the proof can be finished as in the previous paragraph.
\end{proof}

	\section{On intersecting families of graphs of functions}
	Our first aim is to  prove Theorem \ref{degtwo} which we will do separately in the odd and even case.
		
	\begin{lemma}
		\label{kdeg}
		If $U$ is a set of $t$-intersecting polynomials of degree at most $k$ over $\F_q$, then the $(k-t+1)$-ple of coefficients of the monomials $x^t,\ldots,x^{k}$ in elements of $U$ are distinct elements of $\F_q^{k-t+1}$.
	\end{lemma}
	\begin{proof}
		If the coefficients of $x^t,\ldots,x^{k}$ coincide in $f_1,f_2 \in U$, then $f_1-f_2$ would have degree at most $t-1$, and hence at most $t-1$ roots, thus
		the graphs of $f_1$ and $f_2$ would share at most $t-1$ points, a contradiction.
	\end{proof}
	
	Next,  we report below  Lemma 6.1 from \cite{Sam} with an alternative proof.
	\begin{lemma}[{\cite[Lemma 6.1]{Sam}}]
		\label{Sam0}
		Assume $t\leq k < q$. Let $U$ be a set of polynomials of degree at most $k$ over $\F_q$.
		\begin{enumerate}
			\item If for any $f,g\in U$ there exist $x_1,\ldots,x_t \in \F_q$ such that
			$f(x_i)=g(x_i)$ for $i=1,\ldots,t$, then $|U|\leq q^{k-t+1}$.
			\item If for any $f,g\in U$ there are no $x_1,\ldots,x_t \in \F_q$ such that $f(x_i)=g(x_i)$ for $i=1,\ldots,t$, then $|U|\leq q^{t}$.
		\end{enumerate}
	\end{lemma}
	\begin{proof}
		Proof of Part 1. It is a direct consequence of Lemma \ref{kdeg}. 
		
		Proof of Part 2. Take any $t$ distinct field elements, say, $x_1,\ldots,x_t$. For any polynomial $f$ over $\F_q$,
		$(f(x_1),\ldots,f(x_t))$ can take at most $q^t$ distinct values of $\F_q^t$ and hence if $|U|>q^t$ then there will be at least $2$ polynomials in $U$ which have the same values on the set $\{x_1,\ldots,x_t\}$.
	\end{proof}

	%
	%
	%
	%
	
	\begin{lemma}
		\label{equalcdeg2}
		Let $U$ be a set of intersecting polynomials of degree at most $2$ over $\F_q$. Assume that there are more than
		$\lfloor(q+1)/2\rfloor$ polynomials $h_i$ in $U$, so that their $x^2$ coefficients are $c$, for some fixed $c \in \F_q$ and suppose also
		that there exist values $\alpha$ and $\beta$ so that $h_i(\alpha) = \beta$. Then for every polynomial $f \in U$, whose coefficient in $x^2$ is not $c$, $f(\alpha) = \beta$.
	\end{lemma}	
	\begin{proof}
		First assume that $\alpha = 0$. Then the constant term in the polynomials $h_i$ is always $\beta$ and we want to show that for any
		polynomial $f \in U$, if the coefficient of $x^2$ is not $c$, the constant term must be $\beta$. Assume to the contrary, that there is a polynomial $g \in U$, whose constant term is not $\beta$. Consider the polynomials:
		\[\{g-h_i\}.\]
		Since $g$ and $h_i$ are intersecting, $g-h_i$ must have a root in $\F_q$. Also, by the assumptions of the lemma, $(g-h_i)(x) = dx^2 + vx + w$,
		where $d\not=0$ and $w \not=0$ are fixed.  We claim that there are at most $\lfloor(q+1) / 2\rfloor$ such polynomials, hence a contradiction.  Indeed, if
		$(g-h_i)(x)$ has a root in $\F_q$, then it can be written as $d(x-u)(x-\frac{w}{du})$. So to bound the number of possible polynomials, we have to
		bound the number of different $(u, \frac{w}{du})$ pairs, where the order does not matter.
		
		First assume $q$ to be odd. If $w/d$ is not a square, then we get $(q-1)/2$ such pairs.
		If it is a square, then we see $2 + (q-3)/2$ pairs, which is $(q+1)/2$.

		Now, assume  $q$ to be even. In this case the number of different  pairs $(u, \frac{w}{du})$ is $(q-2)/2+1$, that is, $q/2$.
		
		Finally, if $\alpha \not= 0$, then instead of the polynomials $f$ in $U$, consider the polynomials $\bar{f}(x) := f(x+\alpha)$. This new family is clearly
		an intersecting family, the $\bar{h}_i$ polynomials will still have the same leading coefficients and $\bar{h}_i(0)=\beta$, so we are in the previous case.
	\end{proof}
	
	\begin{lemma}
		\label{lotthroughonepoint}
		Let $q$ be even and $U$ be a set of intersecting polynomials of degree at most $2$ over $\F_q$. Assume that
		$|U| > \frac{q^2 + q}{2}$ and assume also that $H$ is a subset of $U$ with more than $\frac{q^2}{2}$ polynomials $h_i$, so that there exist values $\alpha$ and $\beta$ for which $h_i(\alpha) = \beta$.
		Then for every polynomial $f \in U$, $f(\alpha) = \beta$.
	\end{lemma}
	
	\begin{proof}
		By the pigeon hole principle, there exists a value $c$ such that there are more than $\frac{q}{2}$ polynomials in $H$
		with $x^2$ coefficient $c$. Let $U^c$ denote the polynomials in $U$ with $x^2$ coefficient $c$. Then Lemma \ref{equalcdeg2} implies that for any polynomial $f \in (U \setminus U^c)$, $f(\alpha) = \beta$.
		Note that $|U \setminus U^c| > \frac{q^2 + q}{2}-q$.
		Again by the pigeon hole principle, there exists a value $c' \neq c$ so that there are more than $\frac{q^2 - q}{2(q-1)} = \frac{q}{2}$ polynomials in $(U \setminus U^c)$ with $x^2$ coefficient $c'$. Lemma \ref{equalcdeg2} yields that for any polynomial $g \in U^c$, $g(\alpha) = \beta$.
	\end{proof}
	
	The next result can be proved in exactly the same way as Lemma \ref{lotthroughonepoint}.
	
	\begin{lemma}
		\label{lotthroughonepointdispari}
		Let $q$ be odd and $U$ be a set of intersecting polynomials of degree at most $2$ over $\F_q$. Assume that
		$|U| > \frac{q^2 + 2q-1}{2}$ and suppose also that $H$ is a subset of $U$ with more than $\frac{q^2 + q}{2}$ polynomials $h_i$, so that there exist values $\alpha$ and $\beta$ for which $h_i(\alpha) = \beta$. Then for every polynomial $f \in U$, $f(\alpha) = \beta$. \qed
	\end{lemma}
	
	\begin{lemma}
		\label{equalc}
		Let $U$ be a set of intersecting polynomials of degree at most $k>1$ over $\F_q$. Assume that there are more than
		$(q-1)q^{k-2}$ polynomials $h_i$ in $U$, so that their $x^k$ coefficients are $c$, for some fixed $c \in \F_q$ and suppose also
		that there exist values $\alpha$ and $\beta$ so that $h_i(\alpha) = \beta$. Then for every polynomial $f \in U$, whose coefficient of $x^k$ is not $c$, it holds that $f(\alpha) = \beta$.
	\end{lemma}
	
	\begin{proof}
		First assume that $\alpha = 0$. Then the constant term in the polynomials $h_i$ is always $\beta$ and we want to show that for any
		polynomial $f \in U$, if the coefficient of $x^k$ is not $c$, the constant term must be $\beta$. Assume to the contrary, that there is a polynomial $g \in U$, whose constant term is not $\beta$.
		Consider the polynomials:
		\[\{g-h_i\}.\]
		Since $g$ and $h_i$ are intersecting, $g-h_i$ must have a root in $\F_q$. Also, by the assumptions of the lemma, $(g-h_i)(x) = dx^k + v_1x^{k-1} + v_2x^{k-2} + \ldots + v_{k-1}x + w$,
		where $d\not=0$ and $w \not=0$ are fixed. We claim that there are at most $(q-1)q^{k-2}$ such polynomials, hence a contradiction. Indeed, such polynomials can be written in the form $(x-u)(dx^{k-1} + \ldots - w/u)$. Note that $u \not = 0$, because $w\not = 0$, hence $u$ can take $q-1$ values.
		The second term is a polynomial of degree $k-1$, its coefficient in $x^{k-1}$ and its constant term are fixed, so there are at most $q^{k-2}$ different such polynomials.
		
		As before, if $\alpha \not= 0$, then instead of the polynomials $f$ in $U$, consider the polynomials $\bar{f}(x) := f(x+\alpha)$. This new family is clearly
		an intersecting family, the $\bar{h}_i$ polynomials will still have the same leading coefficients and $\bar{h}_i(0)=\beta$, so we are in the previous case.
	\end{proof}

	\subsection{Intersecting families of polynomials of degree at most $2$, over $\F_q$, $q$ odd}

	According to Lemma \ref{kdeg}, the members of an intersecting family of polynomials of degree at most $2$ are of the form $f(b,c)+bx+cx^2$ for some function $f$. More precisely:
	
	\begin{definition}
		\label{function}
		Suppose that $U$ is a set of intersecting polynomials. Put $D=\{(b,c)\in \F_q\times \F_q : a+bx+cx^2 \in U\}$ and define $f \colon D \rightarrow \F_q$ as
		$f(b,c)=a$, where $a\in \F_q$ is the unique field element such that $a+bx+cx^2 \in U$.
	\end{definition}
	

	\begin{lemma}
		Let $U$ be a set of intersecting polynomials of degree at most $2$ and for $b\in \F_q$, $q\geq 53$ odd, define
		\[C_b:=\{ c \in \F_q : f(b,c)+bx+c x^2 \in U\}\]
		and
		\[dom_o:=\{ b \in \F_q : |C_b| > q-\sqrt{q}/2+1/2\}.\]
		There exist functions $s, t \colon dom_o \rightarrow \F_q$ and $h \colon dom_o \rightarrow \{0,1,\ldots,n-1\}$ (where $q=p^n$)
		such that
		for every $c\in C_b$
		\[f(b,c)=s(b)c^{p^{h(b)}}+t(b),\]
		and $-s(b)$ is square in $\F_q$.
	\end{lemma}
	\begin{proof}
		If $f(b,c)+bx+cx^2$ and $f(d,e)+dx+ex^2$ are members of $U$, then the difference of the two polynomials must have a root in $\F_q$ and hence
		\[F(b,c,d,e):=(b-d)^2-4(f(b,c)-f(d,e))(c-e)\]
		is a square. For $c \in C_b$ put $f_b(c)$ for $f(b,c)$.

		For each $b\in \F_q$ and $C,E \in C_b$, $C\neq E$, consider $F(b,C,b,E)=-4(f_b(C)-f_b(E))(C-E)$, which has to be a square, or, equivalently, after dividing by $(C-E)^2$,
		\[-\frac{f_b(C)-f_b(E)}{C-E}\]
		is in $\square_q \cup \{0\}$ for each $C,E\in C_b$.
		
		If $b\in dom_o$, then by Theorem \ref{teonew},  $f_b$ can be uniquely extended  to a function $\tilde{f_b} \colon \F_q \rightarrow \F_q$ determining the same set of directions as $f_b$ and for each $c\in \F_q$  $\tilde{f_b}(c)=s(b)c^{p^{h(b)}}+t(b)$ for some $dom_o \rightarrow \F_q$ functions $s$, $t$ such that $-s(b)$ is a square, and
		a function $h \colon dom_o \rightarrow \{0,1,\ldots,n-1\}$.
		Then for $c\in C_b$ we have ${f_b}(c)=s(b)c^{p^{h(b)}}+t(b)$.
	\end{proof}
	
	\begin{lemma}
		\label{deg2.1}
		If $q \geq 53$ is odd, $U$ is a set of intersecting polynomials of degree at most $2$ such that $|dom_o|>1$, then
		for $b,d \in dom_o$ and $c\in C_b$, $e \in C_d$ recall that
		\[F(b,c,d,e)=(b-d)^2-4(f(b,c)-f(d,e))(c-e)=\]
		\[(b-d)^2-4(s(b)c^{p^{h(b)}}+t(b)-s(d)e^{p^{h(d)}}-t(d))(c-e).\]
		For $b \in dom_o$, one of the following holds
		\begin{enumerate}
			\item $s(b) = s(d) = 0$ and $t(d) =t(b)$ for each $d \in dom_o$,
			\item $s(b) = s(d) \neq 0$, $h(d) = h(b) = 0$ and $(t(b) - t(d))^2 =-s(b)(b-d)^2$ for each $d \in dom_o$,
			\item $s(b) = s(d) \neq 0$, $h(d) = h(b) = n/2$ and $t(b) = t(d)$ for each $d \in dom_o$.
		\end{enumerate}
	\end{lemma}
	\begin{proof}
		
		For $b,d \in dom_o$ and $c\in C_b$, $e \in C_d$ recall that $F(b,c,d,e)$ 
		is a square in $\F_q$.
	
		Define the function $G_{b,d,e} \colon
		\F_q \rightarrow \F_q$, as
		\[c \mapsto (b-d)^2-4(s(b)c^{p^{h(b)}}+t(b)-s(d)e^{p^{h(d)}}-t(d))(c-e)=\]
		\[-4s(b)c^{p^{h(b)}+1}+4es(b)c^{p^{h(b)}}-4(t(b)-s(d)e^{p^{h(d)}}-t(d))c+\]
		\[(b-d)^2+4e(t(b)-s(d)e^{p^{h(d)}}-t(d)).\]
		
		\medskip
		
		\noindent		
		\underline{First assume $0<h(b)<n/2$ for some $b\in dom_o$.}
		\medskip

		Denote by $\psi$ the quadratic character of $\F_q$ and apply Result \ref{res1} to the function $G_{b,d,e}$. Then we have
		\[q-\sqrt{q}-p^{h(b)} < -(q-|C_b|)+(|C_b|-p^{h(b)}-1)\leq \sum_{c\in \F_q}\psi(G_{b,d,e}(c)),\]
 as $G_{b,d,e}$ is a polynomial of degree $p^{h(b)}+1$ in $c$, so the number of its roots  is at most $p^{h(b)}+1$.

		Thus, we cannot have
		\[\left|\sum_{c\in \F_q}\psi(G_{b,d,e}(c))\right|\leq p^{h(b)}\sqrt{q}.\]
		It follows that $G_{b,d,e}$ is the square of a polynomial in $c$.
		And hence by Lemma \ref{smalldegree}, one of the following holds:
		\begin{enumerate}[\rm(i)]
			\item $s(b)=0$ and $t(b)-s(d)e^{p^{h(d)}}-t(d)=0$ for each $d\in dom_o$, $e\in C_d$. If we fix $d$ as well and let $e$ run through $C_d$ then we obtain
			$s(d)=0$ and $t(b)=t(d)$, for each $d\in dom_o$.
			\item $s(b)\neq 0$ and
			\begin{equation}
				\label{eq01}
				( 4 e s(b) )^{p^{h(b)}}(-4s(b))=-4(t(b)-s(d)e^{p^{h(d)}}-t(d))(-4s(b))^{p^{h(b)}},
			\end{equation}
			\begin{equation}
				\label{eq02}
				(4 e s(b))^{p^{h(b)}+1}(-4s(b))=((b-d)^2+4e(t(b)-s(d)e^{p^{h(d)}}-t(d)))(-4s(b))^{p^{h(b)}+1}.
			\end{equation}
			Then $\eqref{eq01}$ yields $s(d)e^{p^{h(d)}}-s(b)e^{p^{h(b)}}=t(b)-t(d)$, for each $d\in dom_o$, $e\in C_d$. Fix $d$ as well and let $e$ run through $C_d$. Put $K$ for the dimension over $\F_p$ of the kernel of the $\F_p$-linear $\F_q \rightarrow \F_q$ function $e \mapsto s(d)e^{p^{h(d)}}-s(b)e^{p^{h(b)}}$. Then
			\[\frac{|C_d|}{p^K} \leq \left|\left\{ s(d)e^{p^{h(d)}}-s(b)e^{p^{h(b)}} : e \in C_d \right\}\right|=|\{ t(b)-t(d) \}|=1,\]
			thus $K=n$ ($q=p^n$) and hence $s(d)=s(b)$, $h(d)=h(b)$
			and $t(d)=t(b)$ for each $d\in dom_o$.
			
			Then \eqref{eq02} reads as $0=(b-d)^2$ for each $d \in dom_o$, a contradiction since $|dom_o|>1$.
			
		\end{enumerate}
		
		We proved that $0<h(b)<n/2$ implies $s(d)=0$ and $t(b)=t(d)$ for each $d\in dom_o$.

		%
		%
		\medskip
		
		\noindent
		\underline{Next assume $n/2<h(b)<n$ for some $b\in dom_o$.}
		\medskip
		
		Apply Result \ref{res1} to the function $c \mapsto (G_{b,d,e}(c))^{p^{n-h(b)}} \pmod{c^q-c}$ and continue as above.
		It turns out that $s(d)=0$ and $t(b)=t(d)$ for each $d\in dom_o$ also in this case.
		
		\medskip
		
		\noindent
		\underline{Now assume $h(b) = n/2$ for some $b \in dom_o$.}
		\medskip
		
		If $s(b)=0$, then
		\[-4(t(b)-s(d)e^{p^{h(d)}}-t(d))c+(b-d)^2+4e(t(b)-s(d)e^{p^{h(d)}}-t(d))\]
		is a square for each $d\in dom_o$ and $c\in C_b$, $e\in C_d$. If we consider $d$  and $e$ fixed as well, then it follows that as a function of $c$ it has to be a constant, so
		$t(b)-s(d)e^{p^{h(d)}}-t(d)=0$ for each $e\in C_d$ and hence $s(d)=0$ and $t(b)=t(d)$ for each $d\in dom_o$.

		If $s(b)\neq 0$, then Lemma \ref{shortcut} applied to $G_{b,d,e}$ gives $\eqref{eq01}$ and hence, as before, $s(d)=s(b)$, $t(d)=t(b)$ and $h(d)=h(b)$ for each $d\in dom_o$.
		
		\medskip
		
		\noindent				
		\underline{Finally, consider the case when $h(b)=0$ for some $b\in dom_o$.}
		\medskip		
		
		Then again from Result \ref{res1}, one obtains $G_{b,d,e}(c)= (b-d)^2-4(s(b)c+t(b)-s(d)e^{p^{h(d)}}-t(d))(c-e)=(\alpha +\beta c)^2$,
		for some $\alpha,\beta \in \F_q$.
		
		If $s(b)=0$, that is, when $G_{b,d,e}$ is a constant, then  $t(b)-s(d)e^{p^{h(d)}}-t(d)=0$ for each $d\in dom_o$, $e\in C_d$, so $s(d)=0$ and $t(b)=t(d)$ for each $d\in dom_o$.
		
		If $s(b)\neq 0$, that is, when $G_{b,d,e}$ is of degree two, then
		the discriminant of $G_{b,d,e}$ has to be zero, i.e.
		\[s(b)(b-d)^2+(s(b)e-s(d)e^{p^{h(d)}}+t(b)-t(d))^2=0.\]
		For $d\in dom_o$ let $\varepsilon_d$ be an element of $\F_q$ for which $\varepsilon_d^2=-s(b)(b-d)^2$.
		Consider $d\in dom_o$ fixed as well, then for $e\in C_d$:
		\[s(b)e-s(d)e^{p^{h(d)}} \in \{\varepsilon_d +t(d)-t(b),-\varepsilon_d +t(d)-t(b)\}.\]
		Put $K$ for the dimension over $\F_p$ of the kernel of the $\F_p$-linear $\F_q \rightarrow \F_q$ function $e \mapsto s(b)e-s(d)e^{p^{h(b)}}$. Then
		\[\frac{|C_d|}{p^K} \leq \left|\left\{ s(b)e-s(d)e^{p^{h(d)}} : e \in C_d \right\}\right|\leq 2,\]
		which is a contradiction, unless $K=n$.
		It follows that $s(b)e-s(d)e^{p^{h(b)}}=0$ for each $e\in \F_q$, so
		$h(d)=0$, $s(d)=s(b)$ and $t(d)-t(b)$ is one of $\varepsilon_d$ and $-\varepsilon_d$.
	\end{proof}

	\begin{lemma}
		\label{deg2}
		If $q \geq 53$ is odd, $U$ is a set of intersecting polynomials of degree at most $2$ such that $|dom_o|>(q+1)/2$, then there exist $\alpha,\beta \in \F_q$ such that $g(\alpha)=\beta$ for all $g\in U$ with $g=f(b,c)+bx+cx^2$ where $b \in dom_o$.
	\end{lemma}
	\begin{proof}
		
		According to Lemma \ref{deg2.1}, we consider the following two cases.
		
		\medskip
		
		\fbox{Suppose that there exists some $b'\in dom_o$ such that $s(b')=0$.}
		
		Then $s(d)=0$ and $t(d)=t(b)$ for each $d\in dom_o$. Put $T$ for $t(b)$.
		It follows that for $b\in dom_o$ and $c\in C_b$ the polynomials
		$f(b,c)+bx+cx^2 \in U$ have the shape  $T+bx+cx^2 \in U$ and hence $(0,T)$ is a common point of their graphs.
		
		\medskip
		
		\fbox{Suppose that $s(b)\neq 0$ for each $b\in dom_o$.}
		
		Then $s(b)=s(d)$ and $h(d)=h(b)$ for each $d\in dom_o$. We will
		denote these values by $S$ and $h$, respectively.
		Note that $h=0$, or $h=n/2$.
		
		\underline{When $h=0$.}
		
		Then $(t(b)-t(d))^2=-S(b-d)^2$ for each $b,d \in dom_o$, and so
		\[\left\{ \frac{t(b)-t(d)}{b-d} : b,d\in dom_o,\, b\neq d \right\} \subseteq \{s,-s\},\]
		where $s^2=-S$.
		
		It follows that the point set $\{(b, t(b)) : b \in dom_o\}$ determines  at most two directions.	But there is no point set determining exactly two directions, thus $t$ determines a unique direction, i.e.
		\[t(d)=\gamma d + T,\mbox{ for } d\in dom_o,\]
		where $\gamma$ is a constant satisfying $\gamma^2=s^2=-S$.
		Then for $b\in dom_o$ and $c\in C_b$ the polynomials
		$f(b,c)+bx+cx^2 \in U$ have the shape $Sc+\gamma b + T+bx+cx^2$, so
		$(-\gamma,T)$ is the common point of their graphs.
		
		\underline{When $h=n/2$.}

		Then also $t(b)=t(d)$ for each $d\in dom_o$.
		Then
		\[(b-d)^2-4S(c-e)^{{\sqrt{q}}+1}\]
		has to be a square for each $b,d\in dom_o$ and $c\in C_b$, $e\in C_d$.
		
		Fix $b,c,d$. Note that for $k\in \F_{\sqrt{q}}\setminus \{0\}$ there are $\sqrt{q}+1$ elements $x$ in $\F_q$ such that $x^{\sqrt{q}+1}=k$. Since $e$ runs through  more than $q-\sqrt{q}/2 + 1/2$ values, we have
		\[\F_{\sqrt{q}}\setminus \{0\} \subseteq \{(c-e)^{\sqrt{q}+1} : e \in C_d\}\]
		so
		\begin{equation}
			\label{final}
			(b-d)^2-Sk=b^2-2db+d^2-Sk
		\end{equation}
		has to be a square in $\F_q$ for each $b,d \in dom_o$, $k\in \F_{\sqrt{q}}\setminus \{0\}$.
		As a polynomial in $b$, the discriminant of \eqref{final} is
		\[4d^2-4(d^2-Sk)=Sk.\]
		Recall $S=s(b)\neq 0$, so this discriminant cannot be zero. 		
		By Result \ref{res1b}, for fixed $d\in dom_o$ and $k\in \F_{\sqrt{q}} \setminus \{0\}$ and for the character $\psi$ of order $2$,
		\[ \sum_{b\in \F_q} \psi(b^2-2db+d^2-Sk)=-\psi(1)=-1.\]
		On the other hand, a lower bound for this sum is
		\[|dom_o |-2-|\F_q\setminus dom_o|,\]
		which is at least $2|dom_o|-q-2$, a contradiction when
		$|dom_o|>(q+1)/2$.
		%
	\end{proof}
	
	Next, we prove Theorem \ref{degtwo} when $q$ is odd.
	
	\begin{theorem}
		\label{deg2dispari}
		If $q \geq 53$ is odd and $U$ is a set of $q^2-\varepsilon$,  $\varepsilon < \frac{q\sqrt{q}}{4} - \frac{3q}{8} - \frac{\sqrt q}{8}$, intersecting polynomials of degree at most $2$ over $\F_q$, then the graphs of the functions in $U$ share a common point.
	\end{theorem}
	\begin{proof}
		Let $dom_o$ denote the set of values as before and let us call a polynomial $f(b,c)+bx+cx^2 \in U$ good if $b \in dom_o$.
		According to the previous lemma, if $|dom_o|>(q+1)/2$ the graphs of the good polynomials share a common point. If there are more
		than $\frac{q^2+q}{2}$ good polynomials then Lemma \ref{lotthroughonepointdispari} finishes the proof.
		
		Clearly,
		\[
		|U| \leq |dom_o|q + (q-|dom_o|)(q-\sqrt{q}/2+1/2) = q^2 - (q-|dom_o|)(\sqrt{q}/2 - 1/2).
		\]
		Assume to the contrary that the number of good polynomials is at most $\frac{q^2+q}{2}$, then $|dom_o| < \frac{q^2+q}{2(q-\sqrt{q}/2+1/2)} < \frac{q}{2} +\frac{\sqrt q}{4} +\frac{1}{2}$. Hence:
		\[
		|U| \leq q^2 -\left(\frac{q\sqrt q}{4} -\frac{3q}{8} -\frac{\sqrt q}{8} + \frac{1}{4}\right),
		\]
		which is a contradiction.
	\end{proof}

	\subsection{Intersecting families of polynomials of degree at most $2$, over $\F_q$, $q$  even }


	\begin{lemma}
		Let $U$ be a set of intersecting polynomials of degree at most $2$ and for $t\in \F_q$, $q>2$ even,  define
		\[B_t:=\{ b \in \F_q : f(b,b+t)+bx+(b+t)x^2 \in U\}\]
		and
		\[dom_e:=\{ t \in \F_q : |B_t|\geq q-\sqrt{q}/2\}.\]
		There exist functions $A,B \colon dom_e \rightarrow \F_q$ such that
		for every $b\in B_t$
		\[f(b,b+t)=A(t)b+B(t),\]
		and $A(t) \in \ker \Tr_{q/2}$.
	\end{lemma}
	\begin{proof}
		Consider $F(x)=f(b,c)+b x + c x^2$ and $G(x)=f(d,e)+d x + e x^2$. Then the graphs of $F$ and $G$ share a common point if and only if $F-G$ has a root in $\F_q$, that is, $b=d$ or 
		\[H(b,c,d,e):=\Tr_{q/2}\left(\frac{(c+e)(f(b,c)+f(d,e))}{(b+d)^2}\right)=0.\]
		Then for each $t\in \F_q$, $b,d \in B_t$, $b\neq d$,
		\[H(b,b+t,d,d+t)=\Tr_{q/2}\left( \frac{(b+d)(f(b,b+t)+f(d,d+t))}{(b+d)^2}  \right)=0.\]
		Simplifying by $b+d$ yields
		\[\Tr_{q/2}\left( \frac{f(b,b+t)+f(d,d+t)}{b+d}  \right)=0.\]
		Define $R_t \colon B_t \rightarrow \F_q$ as $R_t(x)=f(x,x+t)$.
		For each $x,y\in B_t$, $x\neq y$, it holds that
		\[\Tr_{q/2}\left( \frac{R_t(x)+R_t(y)}{x+y}  \right)=0.\]
		In particular, the set of directions determined by the graph of $R_t$ is contained in $\ker \Tr_{q/2}$, and hence it has size at most $q/2$.
		
		From now on assume $t\in dom_e$ and hence $|B_t|\geq q-\sqrt{q}/2$.
		By results of Sz\H onyi \cite{TS}, there exists a unique extension $\tilde{R_t} \colon \F_q \rightarrow \F_q$ of $R_t$ such that the set of directions determined by $\tilde{R_t}$ is the same as the set of directions determined by the point set $\{(x, R_t(x)) : x\in B_t\} \subseteq \AG(2,q)$. So the set of directions determined by $\tilde{R_t}$ is contained in $\ker \Tr_{q/2}$ and hence by Theorem \ref{additiveCarlitz} there exist $A(t), B(t) \in \F_q$ such that $\tilde{R_t}(x)=A(t)x+B(t)$ with $\Tr_{q/2}(A(t))=0$.
		It follows that for $b\in B_t$ we have
		\[R_t(b)=f(b,b+t)=A(t)b+B(t).\]
	\end{proof}
	
	\begin{lemma}
		Let $U$ be a set of intersecting polynomials of degree at most $2$ and define $B_t$, $dom_e$ and the functions $A$ and $B$ as in the previous lemma.
		Then there exist $\alpha, \beta \in \F_q$, $q \geq 8$, such that $A(t)=\alpha^{q/2}+\alpha$ and $B(t)=\alpha t + \beta$ for each $t\in dom_e$.
	\end{lemma}
	\begin{proof}
		If $|dom_e|=1$, then the assertion is trivial, so assume $|dom_e|\geq 2$ and take any $s,t \in dom_e$.
		Fix some $b\in B_s$.
		Then for each $d \in B_t \setminus \{b\}$,
		\[H(b,b+s,d,d+t)=\Tr_{q/2}\left(\frac{(b+s+d+t)(f(b,b+s)+f(d,d+t))}{(b+d)^2}\right)=0,\]
		that is,
		\[\Tr_{q/2}\left(\frac{(b+s+d+t)(f(b,b+s)+A(t)d+B(t))}{(b+d)^2}\right)=0,\]
		i.e.,
		\[\Tr_{q/2}\left(A(t)+\frac{f(b,b+s)+B(t)+A(t)b+A(t)(s+t)}{b+d}+(s+t)\frac{f(b,b+s)+B(t)+A(t)b}{(b+d)^2}\right)=0.\]
		Applying $\Tr_{q/2}(A(t))=0$ and $\Tr_{q/2}(z)=\Tr_{q/2}(z^2)$ for each $z\in \F_q$, we obtain for each $d\in B_t \setminus \{b\}$, $d \neq b$
		\[\Tr_{q/2}\left(\frac{f^2(b,b+s)+B^2(t)+A^2(t)b^2+A^2(t)(s+t)^2+(s+t)(f(b,b+s)+B(t)+A(t)b)}{(b+d)^2}\right)=0.\]
		The numerator does not depend on $d$, while the denominator ranges over a subset of $\F_q^*$ of size
		$|B_t\setminus \{b\}|>\deg \Tr_{q/2} = q/2$ and hence this is possible only if
		\[f^2(b,b+s)+B^2(t)+A^2(t)b^2+A^2(t)(s+t)^2+(s+t)(f(b,b+s)+B(t)+A(t)b)=0.\]
		Since $f(b,b+s)=A(s)b+B(s)$, this is equivalent to
		\[(A(s)b+B(s))^2+B^2(t)+A^2(t)b^2+A^2(t)(s+t)^2+(s+t)(A(s)b+B(s)+B(t)+A(t)b)=0,\]
		that is,
		\[b^2(A^2(s)+A^2(t))+b(A(s)+A(t))(s+t)+(B(s)+B(t))(B(s)+B(t)+s+t)+A^2(t)(t+s)^2=0.\]	
		Since this holds for every $b\in B_s$, and $|B_s|>2$, it follows that as a polynomial of $b$, this is the zero polynomial, so $A(s)=A(t)$ and
		\begin{equation}
			\label{eq1lemma}
			(B(s)+B(t))(B(s)+B(t)+s+t)+A^2(t)(t+s)^2=0.
		\end{equation}
		Since $\Tr_{q/2}(A(t))=0$, and $A(t)$ is a constant function, this proves the existence of $\alpha'\in \F_q$ such that $A(x)=\alpha'^{q/2}+\alpha'$ for each $x\in dom_e$.
		If $|dom_e|=2$, then clearly $B$ is linear, so assume $|dom_e|\geq 3$ and take some $t'\in dom_e \setminus \{s,t\}$. The same arguments show
		\begin{equation}
			\label{eq2lemma}
			(B(s)+B(t'))(B(s)+B(t')+s+t')+A^2(t')(t'+s)^2=0.
		\end{equation}
		Summing up \eqref{eq1lemma} and \eqref{eq2lemma} we obtain
		\[B(s)(t+t')+s(B(t)+B(t'))+B^2(t)+B^2(t')+B(t)t+B(t')t'+(\alpha'^2+\alpha')(t+t')^2=0,\]
		so for $x\in dom_e \setminus \{t,t'\}$ it holds that
		\[B(x)=x \frac{B(t)+B(t')}{t+t'}+\frac{B^2(t)+B^2(t')+B(t)t+B(t')t'}{t+t'}+(\alpha'^2+\alpha')(t+t'),\]
		and from \eqref{eq1lemma} (with $s=t'$) one obtains the same for $x\in \{t,t'\}$, so $B$ is linear.
		Put $B(x)=\gamma x + \beta$, then from \eqref{eq1lemma}
		\[\gamma(s+t)(\gamma(s+t)+(s+t))=(\alpha'^2 + \alpha')(s+t)^2,\]
		so	$\gamma^2+\gamma=\alpha'^2+\alpha'$ which proves $\gamma=\alpha'$ or $\gamma=\alpha'+1$. Now, if $\gamma=\alpha'$ then we set $\alpha:=\alpha'$ whereas if $\gamma=\alpha'+1$ we set $\alpha:=\alpha'+1$. Since $\alpha'^{q/2}+\alpha'=(\alpha'+1)^{q/2}+\alpha'+1$, our lemma follows.
	\end{proof}
	
	\begin{corollary}
		If $g(x)=f(b,c)+bx+cx^2 \in U$ and $b+c \in dom_e$, then $g(\alpha^{q/2})=\beta$.
	\end{corollary}
	\begin{proof}
		Put $t=b+c$. Then $t\in dom_e$ and hence
		\[f(b,c)=f(b,b+t)=A(t)b+B(t)=\alpha^{q/2}b+\alpha b+\alpha t + \beta,\]
		hence
		\[g(\alpha^{q/2})=\alpha^{q/2}b+\alpha b+\alpha t + \beta+b\alpha^{q/2}+(b+t)\alpha=\beta.\]
	\end{proof}

	Finally, we prove Theorem \ref{degtwo} for $q$ even.
	
	\begin{theorem}
		\label{deg2p}
		If $q\geq 8$ is even and $U$ is a set of $q^2-\varepsilon$,  $\varepsilon < \frac{q\sqrt{q}}{4} - \frac{q}{8} - \frac{\sqrt q}{8}$, intersecting polynomials of degree at most $2$ over $\F_q$, then the graphs of the functions in $U$ share a common point.
	\end{theorem}
	\begin{proof}
		Let $dom_e$ denote the set of values as before and let us call a polynomial $f(b,c)+bx+cx^2 \in U$ good if $b+c \in dom_e$.
		According to the previous corollary, the graphs of the good polynomials share a common point. If there are more
		than $\frac{q^2}{2}$ good polynomials then Lemma \ref{lotthroughonepoint} finishes the proof.
		
		Clearly,
		\[
		|U| \leq |dom_e|q + (q-|dom_e|)(q-\sqrt{q}/2) = q^2 - (q-|dom_e|)\sqrt{q}/2.
		\]
		Assume to the contrary that the number of good polynomials is at most $\frac{q^2}{2}$, then $|dom_e| \leq \frac{q^2}{2(q-\sqrt{q}/2)} < \frac{q}{2} +\frac{\sqrt q}{4} +\frac{1}{4}$. Hence:
		\[
		|U| < q^2 -\left(\frac{q\sqrt q}{4} -\frac{q}{8} -\frac{\sqrt q}{8}\right),
		\]
		which is a contradiction.
	\end{proof}

	\subsection{Intersecting families of polynomials of degree at most $k>2$}
	
	{\bf{Theorem \ref{general}.}} {\em If $U$ is a set of more than $q^{k}-q^{k-1}$ intersecting polynomials over $\F_q$, $q \geq 53$ when $q$ is odd and $q \geq 8$ when $q$ is even,
		and of degree at most $k$, $k > 1$, then there exist $\alpha,\beta \in \F_q$ such that $g(\alpha)=\beta$ for all $g\in U$. Furthermore, $U$  can be uniquely extended to a
		family of  $q^k$ intersecting polynomials of degree at most $k$ over $\F_q$.}

	\begin{proof} let $U$ be a set of more than $q^{k}-q^{k-1}$ intersecting polynomials over $\F_q$ and of degree at most $k$, $k > 1$.
		First we show  that there exist $\alpha,\beta \in \F_q$ such that $g(\alpha)=\beta$ for all $g\in U$.
		We prove this by induction.
		
		For $k=2$, this is true by Theorem \ref{degtwo}. Now assume that it is true for $k-1$ and we want to prove it for $k$. By the pigeon hole principle there must be a value $c$, such that there are  more than $q^{k-1}-q^{k-2}$ polynomials $h_i$ in $U$ whose coefficient in $x^k$ is $c$. Now consider the family of polynomials in the form of $\{h_i - cx^k\}$. Clearly, this is an intersecting family of polynomials of degree at most $k-1$. So by the induction hypothesis, there are values $\alpha$ and $\beta$ so that for every $i$, $(h_i - cx^k)(\alpha) = \beta$ and hence of course $h_i(\alpha) = \beta+c\alpha^k$ and so Lemma \ref{equalc} finishes the  proof of the first part.
		
		Next, we will prove that $U$  can be uniquely extended to a family of  $q^k$ intersecting polynomials of degree at most $k$ over $\F_q$.
		Hence, let $\cF$ and $\cF'$ be two intersecting families of size $q^k$, both of them containing $U$.
		Then, there exist $\alpha,\alpha',\beta, \beta' \in \F_q$ such that
		$g(\alpha)=\beta$ for all $g\in \cF$ and $g(\alpha')=\beta'$ for all $g\in \cF'$. The polynomials in $U$ are in $\cF \cap \cF'$, a contradiction unless $(\alpha,\beta)=(\alpha',\beta')$, since there are at most $q^{k-1}<|U|$ polynomials of degree at most $k$, whose graph contains two distinct, fixed points. Theorem  \ref{general} follows.
	\end{proof}

	\section{Large intersecting families whose graphs do not share a common point}
	
	The following construction was drawn to our attention in a talk by Sam Adriaensen.
	Note that it shows the sharpness of the lower bound on $|U|$ in Lemma \ref{lotthroughonepoint}.
	
	\begin{example}[Hilton-Milner type]
		Pick a point $P := (\alpha, \beta)$ and a line $e:= \{(x, vx +w): x \in \F_q\} $ in ${\rm AG}(2, q)$, so that $\beta \neq v\alpha +w$.
		Let $U'$ be the set of those polynomials over $\F_q$, which are of the form $h(x) = cx^2+bx+a$ and for which $h(\alpha) = \beta$ and
		there exist values $\alpha'$ and $\beta'$ so that $h(\alpha') = \beta'$ and $\beta' = v\alpha' +w$. The
		set $U = U' \cup \{e \}$ is a set of intersecting polynomials of degree at most $2$ over $\F_q$.
		The size of $U$ is $\frac{q^2+q}{2}$ and clearly there exist no values $s,t\in \F_q$ so that for every
		polynomial $f \in U$, $f(s) = t$.
	\end{example}
	
	\begin{proof}
		Clearly, we may assume that $P = (0, 1)$ and $e=\{(x, 0): x \in \F_q\}$. Then $a=1$ for the polynomials in $U'$.
		Pick a point $R:=(u, 0)$ from $e$. The number of polynomials $g$ in $U'$, so that $g(u) = 0$ is $0$ if $u=0$, $q$ otherwise. Hence if we count the polynomials of $U'$ corresponding to $R$ when $R$ runs on the points of $e$, we see $q(q-1)$ polynomials. But most of the polynomials in $U'$ will correspond to two different points $R$ and $R'$ of $e$.
		Actually, only the polynomials which are of the form $bx+1$ ($b\in \F_q^*$) and polynomials of the  form $c^{-2}(x+c)^2$ ($c\in \F_q^*$)
		in $U'$ will correspond to exactly one point in $e$. Hence $|U'| = \frac{q(q-1) - 2(q-1)}{2} + 2(q-1) = \frac{q^2+q}{2} - 1$
		and so $|U| =  \frac{q^2+q}{2}$.
	\end{proof}
	
	\begin{example}
		Let $q$ be odd. There is a family $\cM$ of intersecting polynomials of degree at most $2$ such that
		$|\cM|=\frac{q^2-q+1}{2}$ and there exists $f\in \cM$ with the property that $|U_f \cap U_g|=1$ for each $g\in \cM,\,g\neq f$.
	\end{example}
	\begin{proof}
		Choose a polynomial $f(x)=Ax^2+Bx+C$ and let $\square_q$ be the set of non-zero squares in $\F_q$.
		Let
		\[\cP=\left\{a_i x^2 + b_i x +C-\frac{(B-b_i)^2}{4(A-a_i)} : A-a_i\in \square_q ,\, b_i\in \F_q \right\}.\]
		Note that $|\cP|=q(q-1)/2$.
		
		%
		If $(a_i,b_i)$ and $(a_j,b_j)$ correspond to two elements of $\cP$ then the graphs of the corresponding polynomials meet each other if and only if
		\[(b_i-b_j)^2-4(a_i-a_j)\left( \frac{(B-b_j)^2}{4(A-a_j)}-\frac{(B-b_i)^2}{4(A-a_i)}  \right)=\]
		\[ \frac{(a_iB-a_jB-Ab_i+a_jb_i+Ab_j-a_ib_j)^2}{(A-a_i)(A-a_j)}\]
		is a (possibly zero) square in $\F_q$.
		This certainly holds since both $(A-a_i)$ and $(A-a_j)$ are squares.
		Hence $\cP$ is an intersecting family.
		Finally, we prove that $\cM=\cP\cup \{f\}$ is also an intersecting family. We will do this by proving that for each $g\in \cP$, $U_g$ meets $U_f$ in a unique point.
		So assume $g(x)=ax^2+bx+C-\frac{(B-b)^2}{4(A-a)}$.
		It is easy to see that the discriminant of
		$f-g$ is zero and hence the result follows.
	\end{proof}
	
	\subsection*{Acknowledgement}
	
	We are extremely grateful for the reviewer's thorough reading and valuable comments. An inaccuracy spotted out by the reviewer 
	led us to the discovery of Theorem \ref{teonew}.

\end{document}